\documentclass[10pt]{amsart}

\usepackage{amssymb}
\usepackage{amsthm}
\usepackage{amsmath}
\usepackage{mathtools}
\usepackage{soul}
\setstcolor{blue}
\usepackage[usenames]{color}

\usepackage{tikz}
\usetikzlibrary{arrows,calc}

\usepackage[foot]{amsaddr}

\usepackage{hyperref}
\usepackage{cleveref}

\theoremstyle{plain}
\newtheorem{proposition}{Proposition}[section]
\newtheorem{theorem}[proposition]{Theorem}
\newtheorem{lemma}[proposition]{Lemma}
\newtheorem{corollary}[proposition]{Corollary}
\theoremstyle{definition}
\newtheorem{example}[proposition]{Example}
\newtheorem{definition}[proposition]{Definition}

\theoremstyle{remark}
\newtheorem{remark}[proposition]{Remark}
\newtheorem{conjecture}[proposition]{Conjecture}

\DeclareMathOperator{\Aut}{\mathsf{Aut}}
\DeclareMathOperator{\Gsf}{\mathsf{G}}

\DeclareMathOperator{\GL}{\mathsf{GL}}
\DeclareMathOperator{\PU}{\mathsf{PU}}
\DeclareMathOperator{\Usf}{\mathsf{U}}
\DeclareMathOperator{\Ksf}{\mathsf{K}}
\DeclareMathOperator{\Asf}{\mathsf{A}}
\DeclareMathOperator{\Msf}{\mathsf{M}}
\DeclareMathOperator{\Id}{Id}

\DeclareMathOperator{\Cc}{\mathcal{C}}

\DeclareMathOperator{\Oc}{\mathcal{O}}
\DeclareMathOperator{\Pc}{\mathcal{P}}

\DeclareMathOperator{\Uc}{\mathcal{U}}

\DeclareMathOperator{\Bb}{\mathbb{B}}
\DeclareMathOperator{\Cb}{\mathbb{C}}
\DeclareMathOperator{\Db}{\mathbb{D}}

\DeclareMathOperator{\Hb}{\mathbb{H}}

\DeclareMathOperator{\Nb}{\mathbb{N}}

\DeclareMathOperator{\Rb}{\mathbb{R}}
\DeclareMathOperator{\Sb}{\mathbb{S}}

\DeclareMathOperator{\Zb}{\mathbb{Z}}
\DeclareMathOperator{\Qb}{\mathbb{Q}}

\newcommand{\abs}[1]{\left|#1\right|}

\newcommand{\norm}[1]{\left\|#1\right\|}

\newcommand{\vect}[1]{\mathbf{#1}}

\begin{document}

\title[Rigidity of proper holomorphic maps]{Rigidity of proper holomorphic maps between balls with H\"older boundary regularity}

\author[Huang]{Kyle Huang}
\email{kyle.huang@utah.edu}
\author[Park]{Jinwoo Park}
\email{park684@wisc.edu}
\author[Skenderi]{Aleksander Skenderi}
\email{aleksander.skenderi@uj.edu.pl}
\author[Srimurthy]{Jaan Amla Srimurthy}
\email{srimurthy@wisc.edu}
\author[Wen]{Rou Wen}
\email{rwen5@wisc.edu}
\author[Zimmer]{Andrew Zimmer}\address{University of Wisconsin-Madison}
\email{amzimmer2@wisc.edu}
\date{\today}
\subjclass[2010]{}

\begin{abstract} In this paper, we prove a rigidity result for proper holomorphic maps between unit balls that have many symmetries and which extend to H\"older continuous maps on the boundary, with H\"older exponent strictly greater than $1/2$.
\end{abstract}

\maketitle

\section{Introduction}

In this paper we study rigidity results for proper holomorphic maps between Euclidean balls in complex Euclidean space. Classically rigidity results have two forms (i) proper holomorphic maps that are sufficiently smooth at the boundary must be rational, see for instance~\cite{MR1703603,MR969413}, or (ii) when the codimension is small, rational proper holomorphic maps can be characterized, see for instance~\cite{MR857373,MR1872546,MR3157993}. In this paper, we consider a different source of rigidity: the symmetries of the map. 

Given a domain $\Omega \subset \Cb^d$, let $\Aut(\Omega)$ denote the group of biholomorphic maps $\Omega \rightarrow \Omega$. Cartan proved that when $\Omega$ is bounded the group $\Aut(\Omega)$ has a Lie group structure where the map
$$
(\varphi, z) \in \Aut(\Omega) \times \Omega \mapsto \varphi(z) \in \Omega
$$
is smooth. Given two domains $\Omega_1 \subset \Cb^m$ and $\Omega_2 \subset \Cb^M$, the group $\Aut(\Omega_1)\times \Aut(\Omega_2)$ acts on the set of holomorphic maps $\Omega_1 \rightarrow \Omega_2$ by 
$$
(\phi,\varphi) \cdot f  := \varphi \circ f \circ \phi^{-1}.
$$
Then, given a holomorphic map $f : \Omega_1 \rightarrow \Omega_2$, the \emph{symmetry group of $f$}, denoted $\Gsf_f$, is the stabilizer of $f$ under this action, i.e. 
$$
\Gsf_f : = \left\{ (\phi,\varphi) \in \Aut(\Omega_1)\times \Aut(\Omega_2) : f = \varphi \circ f \circ \phi^{-1}\right\}. 
$$

\begin{example}\label{ex:trivial_example} Given $M \geq m \geq 1$, consider the trivial proper holomorphic map $f : \Bb^m \rightarrow \Bb^M$ between Euclidean balls $\Bb^m \subset \Cb^m$, $\Bb^M \subset \Cb^M$ given by 
$$
f(z): = (z,\vect{0}).
$$
Then the projection $\Gsf_f \rightarrow \Aut(\Bb^m)$ is onto and so, in this case, the group $\Gsf_f$ is very large.
\end{example}

We consider the following imprecise conjecture.

\begin{conjecture} If $f : \Bb^m \rightarrow \Bb^M$ is a proper holomorphic map, $\Gsf_f$ is ``sufficiently large,'' and $f$ extends to a ``sufficiently regular'' map $\overline{\Bb^m} \rightarrow \overline{\Bb^M}$, then up to composition with automorphisms $f$ is the map in Example~\ref{ex:trivial_example}.
\end{conjecture}

Previously:
\begin{enumerate}
\item D'Angelo--Xiao~\cite[Corollary 3.2]{DX2017}  proved that the conjecture is true when $\Gsf_f$ is non-compact and $f$ is rational (note that a rational proper map extends to a $\Cc^\infty$-smooth map $\overline{\Bb^m} \rightarrow \overline{\Bb^M}$~\cite{MR1047119}). 
\item Gevorgyan--Wang--Zimmer~\cite{GWZ2004} proved that it suffices to assume that  $\Gsf_f$ is non-compact and  $f$ extends to a $\Cc^2$-smooth map $\overline{\Bb^m} \rightarrow \overline{\Bb^M}$. 
\end{enumerate} 

In this paper we lower the required boundary regularity to $\alpha$-H\"older for some $\alpha > 1/2$, but increase the symmetry assumption.

\begin{theorem}\label{thm:main} If $2 \leq m < M$ and $f: \Bb^m \rightarrow \Bb^M$ is a proper holomorphic map where 
\begin{enumerate} 
\item the image of the projection $\Gsf_f \rightarrow \Aut(\Bb^m)$ is Zariski dense,
\item $f$ extends to an $\alpha$-H\"older continuous map $\overline{\Bb^m} \rightarrow \overline{\Bb^M}$ for some $\alpha > \frac{1}{2}$, 
\end{enumerate} 
then there exist $\phi_1 \in \Aut(\Bb^M)$ and $\phi_2 \in \Aut(\Bb^m)$ such that  
\begin{align*}
\phi_1 \circ f \circ \phi_2(z) = (z,\vect{0})
\end{align*}
for all $z \in \Bb^m$. 
\end{theorem}

The group $\Aut(\Bb^m)$ is a simple Lie group and hence has the structure of a real algebraic group, so there is a well-defined notion of Zariski dense subgroups. Using the explicit structure of $\Aut(\Bb^m)$ we provide an elementary definition of Zariski dense subgroups in Section ~\ref{sec:Zariski-Density} below. 

We also note that Zariski density is generic in the following sense: If $\mu$ is a probability measure on $\Aut(\Bb^m)$ induced by a smooth volume form (recall, $\Aut(\Bb^m)$ is a Lie group), then for $\mu\otimes \mu$-a.e. $(g,h) \in \Aut(\Bb^m) \times \Aut(\Bb^m)$ the group generated by $g,h$ is Zariski dense (see~\cite[Lemma 6.8]{2008arXiv0804.1395B}). 

Theorem~\ref{thm:main} is related to a well-known conjecture involving complex hyperbolic $m$-space, denoted $\Hb^m_{\Cb}$ (see for instance~\cite[Problem 3.2]{Huang_survey}). 

\begin{conjecture}\label{conj:famous conj} If $2 \leq m \leq M$ and $\rho : \Gamma \rightarrow \mathsf{Isom}(\Hb_{\Cb}^M)$ is a convex co-compact representation of a uniform lattice $\Gamma < \mathsf{Isom}(\Hb^m_{\Cb})$, then the image of $\rho$ preserves a totally geodesic copy of $\Hb^m_{\Cb}$ in $\Hb_{\Cb}^M$.
\end{conjecture}

Using the work of Cao--Mok~\cite{CaoMok1990}, Yue~\cite{Yue1996} proved this conjecture in the special case when $M \leq 2m-1$.

Complex hyperbolic $m$-space is biholomorphic to the unit ball $\Bb^m$ and under this identification $\Aut(\Bb^m)$ coincides with $\mathsf{Isom}_0(\Hb^m_{\Cb})$, the connected component of the identity in $\mathsf{Isom}(\Hb^m_{\Cb})$. 

Now suppose that $\rho: \Gamma \rightarrow \mathsf{Isom}(\Hb_{\Cb}^M)$ is as in Conjecture~\ref{conj:famous conj}. Since $\Aut(\Bb^m) = \mathsf{Isom}_0(\Hb^m_{\Cb})$ has index two in $\mathsf{Isom}(\Hb^m_{\Cb})$, we can replace $\Gamma$ by an index two subgroup and assume that $\Gamma < \Aut(\Bb^m)$. The theory of harmonic maps implies the existence of a $\rho$-equivariant proper holomorphic map $f: \Bb^m \rightarrow \Bb^M$ (see~\cite[pg. 348]{Yue1996}). One can show that $\rho$ preserves a totally geodesic copy of $\Hb^m_{\Cb}$ in $\Hb_{\Cb}^M$ if and only if $f$ satisfies the conclusion of Theorem~\ref{thm:main}. Further, since $f$ is $\rho$-equivariant, 
$$
 \{ (\gamma, \rho(\gamma)) : \gamma \in \Gamma\} < \Gsf_f
$$
and hence the image of the projection $\Gsf_f \rightarrow \Aut(\Bb^m)$ contains $\Gamma$. Finally, since the orbit map of any convex co-compact representation is a quasi-isometry, one can show that the map $f$ extends to a H\"older continuous map $\overline{\Bb^m} \rightarrow \overline{\Bb^M}$.

Hence, Conjecture~\ref{conj:famous conj} can be restated as follows:

\begin{conjecture} Suppose $f: \Bb^m \rightarrow \Bb^M$ is a proper holomorphic map which extends to H\"older continuous map $\overline{\Bb^m} \rightarrow \overline{\Bb^M}$. If the image of the natural projection $\mathsf{G}_f \rightarrow \Aut(\Bb^m)$ contains a uniform lattice, then there exist $\phi_1 \in \Aut(\Bb^M)$ and $\phi_2 \in \Aut(\Bb^m)$ such that 
$$
\phi_1 \circ f \circ \phi_2(z) = (z,\vect{0})
$$
for all $z \in \Bb^m$. 
\end{conjecture}

A lattice is always Zariski dense (by the Borel density theorem) and so Theorem~\ref{thm:main} provides a proof of the conjecture when the H\"older exponent  is strictly larger than $1/2$. 

\subsection{Outline of the proof}\label{sec:outline of proof} The key tool in the proof of Theorem~\ref{thm:main} is the Fundamental Theorem of Affine Geometry. 

Recall that a fractional linear transformation is a map $f : \Uc \rightarrow \Cb^m$ of the form 
$$
f(z) = \frac{Az+b}{c^T z + d}
$$
where  $A$ is an $m$-by-$m$ matrix; $c,b \in \Cb^m$; $d \in \Cb$; and 
$$
\Uc: =\{ z\in \Cb^m :c^T z + d \neq 0\}.
$$

Given such a map, if $L \subset \Cb^m$ is a complex affine line (i.e. $L = a + \Cb \cdot b$ for some $a,b \in \Cb^m$ with $b \neq 0$), then it is straightforward to show that $f(L \cap \Uc)$ is contained in a complex affine line. The Fundamental Theorem of Affine Geometry states that fractional linear transformations are essentially the only maps with this property. 

Given a unital homomorphism $\tau : \Cb \rightarrow \Cb$, let $\tau_m : \Cb^m \rightarrow \Cb^m$ denote the map obtained by applying $\tau$ to each coordinate, i.e.
$$
\tau_m(z)= (\tau(z_1),\dots, \tau(z_m)). 
$$

\begin{theorem}[{The Fundamental Theorem of Affine Geometry, see e.g.~\cite[Theorem 3.5]{McCallum}}]\label{thm:FTAG} Suppose $m \geq 2$ and $\Uc \subset \Cb^m$ is an open subset. Assume $f : \Uc \rightarrow \Cb^m$ is a map with the following properties: 
\begin{enumerate} 
\item The image $f(\Uc)$ contains $(m+1)$-points in general affine position (i.e. they are not contained in a  $(m-1)$-dimensional affine plane).  
\item For every complex affine line $L \subset \Cb^m$, the image $f(L \cap \Uc)$ is contained in a complex affine line. 
\end{enumerate}
Then there exist a unital homomorphism $\tau : \Cb \rightarrow \Cb$ and 
\begin{align*}
g = \begin{pmatrix} A & {b} \\ c^T & d \end{pmatrix}  \in \mathsf{GL}(m+1,\Cb)
\end{align*}
such that 
$$
f(z) =  \frac{A\tau_m(z)+b}{c^T\tau_m(z)+d}
$$
and 
$$
\left\{ z : c^T\tau_m(z) + d=0\right\} \cap \Uc = \emptyset.
$$
\end{theorem} 

\begin{remark} The only measurable unital homomorphisms $\Cb \rightarrow \Cb$ are the identity and complex conjugation.

\end{remark}

We will show that under the hypotheses of Theorem~\ref{thm:main}, the proper holomorphic map sends affine lines into affine lines and thus by the  Fundamental Theorem of Affine Geometry $f$ is rational. Once rationality is established we can appeal to the earlier work D'Angelo--Xiao~\cite[Corollary 3.2]{DX2017} to complete the argument. 

Recall that an element of $\Aut(\Bb^m)$ is either elliptic, parabolic, or loxodromic (see Section~\ref{sec:background automorphism group}). A loxodromic element $g$ has exactly two fixed points in $\partial \Bb^m$ and these fixed points can be labeled $x^\pm_g$ such that 
$$
\lim_{n \rightarrow \pm \infty} g^n(z) = x^\pm_g
$$
for all $z \in \overline{\Bb^m} \smallsetminus \{x^\mp_g\}$. In fact, one can compute the rate of convergence (see Theorem~\ref{thm: contraction rates of loxodromics} below). Let $L_g$ denote the complex affine line containing $x^\pm_g$. Then there exists $\lambda > 1$ (which depends on $g$) such that: if $z \in \overline{\Bb^m} \smallsetminus  \{x_{g}^{+}, x_{g}^{-}\}$, then 
$$
\lim_{n \rightarrow \infty} \frac{1}{n} \log \|g^n(z) - x^+_g\| = 
\begin{cases} 
-2\log \lambda  & \text{if } z \in L_g \\
-\log \lambda & \text{if } z \notin L_g 
\end{cases}. $$

Using this multiplicative gap in exponential contraction rates, we will prove the following. 

\begin{theorem}[see Theorem~\ref{thm:loxo lines are mapped to loxo lines} below]\label{thm:loxo lines are mapped to loxo lines - intro} Suppose $f : \Bb^m\rightarrow \Bb^M$ is a proper holomorphic map that extends to an $\alpha$-H\"older map $\overline{\Bb^m} \rightarrow \overline{\Bb^M}$ for some $\alpha > 1/2$.  If $(\phi,\psi) \in \Gsf_f$ and $\phi$ is loxodromic, then $\psi$ is loxodromic. Moreover, 
$$
f\left(L_{\phi} \cap \overline{\Bb^m} \right) \subset L_{\psi} \cap \overline{\Bb^M}.
$$
\end{theorem} 

Theorem~\ref{thm:loxo lines are mapped to loxo lines - intro} says that certain complex affine lines are mapped into complex affine lines. To upgrade this to all complex affine lines we use the Zariski density assumption. Suppose $f$ satisfies the hypotheses of Theorem~\ref{thm:main} and for $x,y \in \partial \Bb^m$ distinct let $L_{xy}$ be the complex affine line containing $x,y$. To prove that $ f\left(L_{xy} \cap \overline{\Bb^m}\right)$ is contained in a complex affine line for all such $x,y$, we argue as follows:
\begin{enumerate}
\item Let $\Gamma$ be the image of the projection $\Gsf_f \rightarrow \Aut(\Bb^m)$. Using the density of loxodromic fixed points in the limit set $\Lambda(\Gamma)$ of $\Gamma$, we show that if $x,y \in \Lambda(\Gamma)$ are distinct,  then $ f\left(L_{xy} \cap \overline{\Bb^m}\right)$ is contained in a complex affine line (see Corollary~\ref{cor:complex lines intersecting the limit set}). 
\item Using the fact that $f$ is holomorphic, we prove that if $x \in \partial \Bb^m$, then the set 
\[
Z_{x} := \left\{y \in \partial\Bb^m\smallsetminus\{x\} : f\left(L_{xy} \cap \overline{\Bb^m}\right) \text{ is contained in a complex affine line}\right\}
\]
 is a real analytic subset (see Theorem~\ref{thm: Z is a real analaytic variety}). 
\item Using the Zariski density assumption we prove: if $x \in \partial \Bb^m$ and $Z \subset \partial\Bb^m\smallsetminus\{x\}$ is a real analytic subset with $\Lambda(\Gamma) \smallsetminus\{x\} \subset Z$, then $Z = \partial\Bb^m\smallsetminus\{x\}$ (see Theorem~\ref{thm:analytic envelopes}). 
\item Using (2) and (3), we show that  $ f\left(L_{xy} \cap \overline{\Bb^m}\right)$ is contained in a complex affine line for all $x,y$ distinct (see Section~\ref{sec:final section}). 
\end{enumerate} 

The proof of (3) uses the contracting properties of loxodromic elements to rescale $Z$ to obtain a real algebraic subset containing the limit set. 

\subsection*{Acknowledgments} We thank the referee for their helpful and detailed comments which improved the paper.

This paper is the result of a Spring 2025 Madison Experimental Mathematics Lab project at the University of Wisconsin-Madison. Srimurthy, Huang, and Park were undergraduate researchers; Skenderi and Wen were graduate student mentors; and Zimmer was the faculty mentor. We thank the Department of Mathematics at the University for supporting this program. 

Huang and Skenderi were partially supported by grant DMS-2230900 from the National Science Foundation. Wen was partially supported by grant DMS-2105580 from the National Science Foundation. Zimmer was partially supported by a Sloan research fellowship and grants DMS-2105580 and DMS-2452068 from the National Science Foundation.

\section{Background}

\subsection{Possibly ambiguous notation} 
\begin{itemize}
\item For convenience, we use row and column vectors interchangeably. 
\item When $z \in \Cb$, $\norm{z}$ denotes the standard $\ell^2$-norm on complex Euclidean space. When $X$ is a matrix with complex entries, $\norm{X}$ denotes the Frobenius norm.
\item $e_1,\dots, e_m$ denotes the standard basis of $\Cb^m$. 
\item We always identify $\Aut(\Bb^m)=\PU(m,1)$, see Section~\ref{sec:background automorphism group}.

\end{itemize}

\subsection{The automorphism group of the unit ball}\label{sec:background automorphism group}

In this section we recall some basic results about the automorphism group of the unit ball. For more details, see~\cite[Section 2.2]{Abate1989}.

Let $\Usf(m,1)< \GL(m+1,\Cb)$ be the group which preserves the indefinite inner product 
$$[v,w]_{m,1}:= v_1 \overline{w_1} +\cdots +v_m \overline{w_m} -v_{m+1} \overline{w_{m+1}}
$$
on $\mathbb{C}^{m+1}$. 
In this paper, we identify the projectivization 
$$
\PU(m,1) = \Usf(m,1) / \Sb^1 \cdot \Id
$$
with $\Aut(\Bb^m)$ via the action $\PU(m,1) \curvearrowright  \Bb^m$ given by
$$
g(z) = \frac{Az+b}{{c^T}z+d}
$$
where 
\begin{align*}
g = \begin{bmatrix} A & {b} \\ {c^T} & d \end{bmatrix} \in \PU(m,1)
\end{align*}
with $A$ an $m$-by-$m$ matrix; $c,b \in \Cb^m$; and $d \in \Cb$. 

Next we introduce some special subgroups of $\Aut(\Bb^m)$. Let 
$$
\Asf:= \{ a_t : t \in \Rb\} < \Aut(\Bb^m)
$$
where 
\begin{align*}
a_t := \begin{bmatrix} \cosh(t) & 0 & \sinh(t) \\ 0 & \Id_{m-1} & 0 \\ \sinh(t) & 0 & \cosh(t) \end{bmatrix}.
\end{align*}

Also, let $\Ksf < \Aut(\Bb^m)$ denote the stabilizer of $\vect{0}$ in $\Aut(\Bb^m)$. One can show that
\begin{align*}
 \Ksf = \left\{ \begin{bmatrix} U & 0 \\ 0 & 1 \end{bmatrix} : U \in \Usf(m) \right\} < \Aut(\Bb^m),
\end{align*} 
for a proof see \cite[Corollary 2.2.2]{Abate1989}. Also, notice that $\Ksf$ preserves the $\ell^2$-norm, i.e. $||k(v)||= ||v||$ for all $v \in \Cb^m$ and $k \in \Ksf$.

Finally, let $\Msf$ be the centralizer of $\Asf$ in $\Ksf$, that is 
\begin{align*}
\Msf = \left\{ \begin{bmatrix} 1 & 0 & 0 \\ 0& U & 0 \\ 0 & 0 & 1 \end{bmatrix} : U \in \Usf(m-1) \right\} < \Aut(\Bb^m).
\end{align*}

Any element $g \in \Aut(\Bb^{m})$ extends to a smooth map $\overline{\Bb^{m}} \rightarrow \overline{\Bb^{m}}$, and so by the Brouwer fixed point theorem this extension has at least one fixed point in $\overline{\Bb^{m}}$. In fact, every automorphism of the ball is exactly one of the following three types.

\begin{theorem}[{see e.g. \cite[Proposition 2.2.9]{Abate1989}}]
An element $g \in \Aut(\Bb^{m})$ is either:
\begin{itemize}
    \item[(a)] $\textit{elliptic}$, i.e. it fixes at least one point in $\Bb^{m}$, 
    \item[(b)] $\textit{parabolic}$, i.e. it is not elliptic and fixes exactly one point in $\partial \Bb^{m}$, or
    \item[(c)] $\textit{loxodromic}$, i.e. it is not elliptic and fixes exactly two points in $\partial \Bb^{m}$.  
\end{itemize}
\end{theorem} 

Loxodromic elements have the following dynamical behavior.

\begin{theorem}\label{thm:structure of loxodromic} If $g \in \Aut(\Bb^m)$ is loxodromic, then $g = h ka_t h^{-1}$ for some $h \in \Aut(\Bb^m)$, $k \in \Msf$, and $t  > 0$. Moreover, the fixed points of $g$ are $x_g^\pm = h(\pm e_1)$ and 
$$
\lim_{n \rightarrow \pm \infty} g^{n}(z)= x_g^\pm
$$
for all $z \in \overline{\Bb^m} \smallsetminus \{x_g^\mp \}$ (in fact, the convergence is uniform on compact subsets of $\overline{\Bb^m} \smallsetminus \{x_g^\mp \}$). 
 \end{theorem}

 \begin{proof} A proof of the first assertion can be found in~\cite[Proposition 2.2.10]{Abate1989}. Since $ka_t$ has fixed points $\pm e_1$, we see that $g$ has fixed points $x_g^\pm = h(\pm e_1)$. To establish the last assertion, it suffices to consider the case when $g = ka_t$ for some $t > 0$. Then the claim follows from a computation. 
 \end{proof} 
 
 \begin{remark} We will always assume that the fixed points of a loxodromic element are labeled as in Theorem~\ref{thm:structure of loxodromic}.\end{remark}
 
 We will also use the so-called $\mathsf{KAK}$-decomposition. 
 
 \begin{theorem}\label{thm:KAK} Any $g \in \Aut(\Bb^m)$ can be written as $g = k_1 a_t k_2$ for some $k_1,k_2 \in \Ksf$ and $t \geq 0$. \end{theorem}
 
Since the proof is short we include it here. 

\begin{proof} Since $\Ksf$ acts transitively on spheres centered at $\vect{0}$, we can find $k_1 \in \Ksf$ such that $k_1^{-1}g(\vect{0}) = (\norm{g(\vect{0})}, 0,\dots, 0)$. Since 
$$
\{a_t(\vect{0}) : t \geq 0\} = \{ (\tanh t, 0,\dots, 0) : t \geq 0\}=[0,1)\times \{(0,\dots,0)\},
$$
we can pick $t \geq 0$ such that $a_{t}^{-1} k_1^{-1} g(\vect{0}) = \vect{0}$. Since $\Ksf$ is the stabilizer of $\vect{0}$ in $\Aut(\Bb^m)$,  we then have $a_{t}^{-1} k_1^{-1} g=k_2$ for some $k_2 \in \Ksf$, equivalently $g = k_1 a_t k_2$.
\end{proof}

\subsection{The limit set and north-south dynamics}

The \emph{limit set}, denoted $\Lambda(\Gamma)$, of a subgroup $\Gamma < \Aut(\Bb^m)$ is the set of accumulation points in $\partial \Bb^m$ of the orbit $\Gamma \cdot x$ for some (any) $x \in \Bb^m$. In the literature, one normally encounters limit sets of discrete subgroups of Lie groups, but we will consider the limit sets of more general (not necessarily discrete) subgroups of $\Aut(\Bb^m)$. Notice that $\Lambda(\Gamma)$ is a $\Gamma$-invariant subset of $\partial \Bb^m$. 

We will use the following result on the north-south-like dynamics of the action of $\Aut(\Bb^m)$ on $\overline{\Bb^m}$.

\begin{theorem}\label{thm:North South} If $\{g_n\} \subset \Aut(\Bb^m)$, $g_n(\vect{0}) \rightarrow x \in \partial \Bb^m$, and $g_n^{-1}(\vect{0}) \rightarrow y\in \partial \Bb^m$, then 
$$
\lim_{n \rightarrow \infty} g_n(z) =x
$$
for all $z \in \overline{\Bb^m} \smallsetminus \{y\}$, and the convergence is uniform on compact subsets of $\overline{\Bb^m} \smallsetminus \{y\}$.
\end{theorem} 

This is well-known, but since the proof is short we include it here. 

\begin{proof} Using Theorem~\ref{thm:KAK}, we can write $g_n = k_{1,n} a_{t_n} k_{2,n}$ where $k_{1,n}, k_{2,n} \in \Ksf$ and $t_n \geq 0$. Since 
$$
1 = \lim_{n \rightarrow \infty} \norm{g_n(\vect{0})} =  \lim_{n \rightarrow \infty}\norm{a_{t_n}(\vect{0})} 
$$
we have $t_n \rightarrow +\infty$. Notice that 
\begin{equation}\label{eqn:action of at}
\lim_{t \rightarrow \pm \infty} a_{t}(z) = \lim_{t \rightarrow \pm\infty} \frac{(\cosh(t) z_1 + \sinh(t), z_2, \dots, z_m)}{\sinh(t)z_1 + \cosh(t)}= \pm e_1
\end{equation} 
for all $z \in \overline{\Bb^m} \smallsetminus \{\mp e_1\}$, with the convergence being uniform on compact subsets of $\overline{\Bb^m} \smallsetminus \{\mp e_1\}$. Then, since $\Ksf$ acts isometrically on $\partial \Bb^m$, 
$$
x= \lim_{n \rightarrow \infty} k_{1,n} a_{t_n}(\vect{0}), \quad \text{and} \quad y= \lim_{n \rightarrow \infty} k_{2,n}^{-1} a_{-t_n}(\vect{0})
$$
we have $k_{1,n}(e_1) \rightarrow x$ and $k_{2,n}^{-1}(-e_1) \rightarrow y$. These limits combined with Equation~\eqref{eqn:action of at} imply the result. 
\end{proof}

\subsection{Zariski Density}\label{sec:Zariski-Density}
 In this section, we provide an elementary definition of Zariski dense subgroups of $\Aut(\Bb^m)$. 

Given a real or complex vector space $V$, a function $P : V \rightarrow \Rb$ is a \emph{real polynomial} if for some (any) $\Rb$-basis $b_1,\dots, b_n$ of $V$,  the map 
$$
(x_1,\dots, x_n) \in \Rb^n \mapsto P(x_1 b_1 + \cdots + x_n b_n) \in \Rb
$$
is a polynomial in $x_1,\dots, x_n$.

In the next definition, let $M_{m+1}(\Cb)$ denote the vector space of $(m+1)$-by-$(m+1)$ complex matrices. 

\begin{definition}\label{def:Zariski density} Given a subgroup $H < \Aut(\Bb^m)=\PU(m,1)$, let $\tilde{H}$ denote the preimage of $H$ in $\Usf(m,1)$ under the map $g \in \Usf(m,1) \mapsto [g] \in \PU(m,1)$. Then,  $H$ is \emph{Zariski dense} if whenever $P : M_{m+1}(\Cb) \rightarrow \Rb$ is a real polynomial with $P|_{\tilde H} \equiv 0$, then $P|_{\Usf(m,1)} \equiv 0$. 
\end{definition}

We record some well-known facts about Zariski dense subgroups in semisimple Lie groups. Since the proofs are much more elementary for $\PU(m,1)$ we include them.

\begin{proposition}\label{prop:properties Z dense} Suppose $\Gamma < \Aut(\Bb^m)$ is Zariski dense. Then: 
\begin{enumerate}
\item  $\Lambda(\Gamma) \neq \emptyset$. 
\item If $P : \Cb^m \rightarrow \Rb$ is a real polynomial, $z_0 \in \partial \Bb^m$, and $P|_{\Gamma \cdot z_0} \equiv 0$, then $P|_{\partial \Bb^m} \equiv 0$.
\item If $z_0 \in \partial \Bb^m$, then the orbit $\Gamma \cdot z_0$ is infinite. 
\item If $P : \Cb^m \rightarrow \Rb$ is a real polynomial and $P|_{\Lambda(\Gamma)} \equiv 0$, then $P|_{\partial \Bb^m} \equiv 0$.
\end{enumerate}
\end{proposition} 

\begin{proof} (1). Suppose not. Then $\Gamma \cdot \vect{0}$ is a subset of a compact set $C \subset \Bb^m$. This implies that $\Gamma$ has a fixed point $z_0 \in \Bb^m$, see~\cite[Theorem 12.2]{Frankel1989}. Since $\Aut(\Bb^m)$ acts transitively on $\Bb^m$, see ~\cite[Corollary 2.2.2]{Abate1989}, and  $\Ksf$ is the stabilizer of $\vect{0}$ in $\Aut(\Bb^m)$,  we then have $\Gamma < h_0 \Ksf h_0^{-1}$ for some $h_{0} \in \Aut(\Bb^m)$. Then $\tilde \Gamma <h \Usf(m+1)h^{-1}$ for some $h \in \GL(m+1,\Cb)$. Hence
$$
\tilde \Gamma < \{ g \in \GL(m+1,\Cb) : (h^{-1} gh)^* h^{-1} g h =\Id\}
$$
Then 
$$
P(X) := \norm{ (h^{-1} Xh)^* h^{-1} X h - \Id}^2
$$
is a real polynomial that vanishes on $\tilde \Gamma$, but not on $\Usf(m,1)$. So we have a contradiction. 

(2). Let 
$$
\Uc : = \left\{\begin{pmatrix} A & {b} \\ c^T & d \end{pmatrix} \in M_{m+1}(\Cb) :  c^T z_0 + d \neq 0 \right\}.
$$
Notice that if $g = \begin{pmatrix} A & {b} \\ c^T & d \end{pmatrix} \in \Usf(m,1)$, then the linear action (recall we identify column and row vectors) yields 
$$
g(z_0,1) = (Az_0+b, c^T z_0 + d).
$$
Hence, 
$$
0=[(z_0,1), (z_0,1)]_{m,1} = [g(z_0,1), g(z_0,1)]_{m,1} = \norm{Az_0+b}^2 - \abs{c^T z_0 + d}^2.
$$
Since $g \in \GL(m+1,\Cb)$, the vector $g(z_0,1)$ is non-zero. So we must have $c^T z_0 + d \neq 0$ and so $g \in \Uc$. Now consider the function $R : \Uc \rightarrow \Rb$ given by 
$$
R\left( \begin{pmatrix} A & {b} \\ c^T & d \end{pmatrix}\right): = P\left(\frac{Az_0+b}{c^T z_0 + d} \right).
$$
Then there exist real polynomials $P_1, P_2 : M_{m+1}(\Cb) \rightarrow \Rb$ such that $R(X) =\frac{P_1(X)}{P_2(X)}$ and $P_2(X) \neq 0$ for all $X \in \Uc$. Since $P|_{\Gamma \cdot z_0} \equiv 0$, we have $P_1|_{\tilde \Gamma} \equiv 0$. Hence $P_1|_{\mathsf{U}(m,1)} \equiv 0$. 

Since $\PU(m,1) = \Aut(\Bb^m)$ acts transitively on $\partial \Bb^m$ (in fact the subgroup $\Ksf$ acts transitively on $\partial \Bb^m$), we then have $P|_{\partial \Bb^m} \equiv 0$.

(3). Suppose for a contradiction that $\Gamma \cdot z_0 = \{ z_1,\dots, z_n\}$. Then 
$$
P(z) := \prod_{j=1}^n \norm{ z-z_j}^2
$$
is a real polynomial which vanishes on $\Gamma \cdot z_0$, but not on $\partial \Bb^m$. So we have a contradiction with (2). 

(4). Since $\Lambda(\Gamma)$ is $\Gamma$-invariant, this follows immediately from (2). 
\end{proof}

\subsection{Real analytic subsets}

In this section we recall the definition and some properties of real analytic subsets. 

Given an open set $\Uc \subset \Rb^d$, a relatively closed subset $Z \subset \Uc$ is a \emph{real analytic subset of $\Uc$} if for every $x \in Z$ there exist an open neighborhood $\Oc$ of $x$ in $\Uc$ and a real analytic function $h : \Oc \rightarrow \Rb$ with 
$$
Z \cap \Oc = h^{-1}(0).
$$

Using the local Noetherian property for real analytic functions, we have the following. 

\begin{theorem}\label{thm:noetherian property} Suppose $\Uc \subset \Rb^d$ is an open set  and $Z \subset \Uc$ is relatively closed in $\Uc$. If for every $x \in Z$ there exist an open neighborhood $\Oc$ of $x$ in $\Uc$ and real analytic functions $\{h_j : \Oc \rightarrow \Rb\}_{j \in \Nb}$ such that 
$$
Z \cap \Oc = \bigcap_{j=1}^\infty h_j^{-1}(0),
$$
then $Z$ is a real analytic subset of $\Uc$. 
\end{theorem}

\begin{proof} Fix $x \in Z$, an open neighborhood $\Oc$ of $x$ in $\Uc$, and real analytic functions $\{h_j : \Oc \rightarrow \Rb\}_{j \in \Nb}$ such that 
$$
Z \cap \Oc = \bigcap_{j=1}^\infty h_j^{-1}(0).
$$
Next fix an open neighborhood $\Oc' \subset \Oc$ of $x$ with $\overline{\Oc'} \subset \Oc$. Then by~\cite[Chapter V, Corollary 1]{MR217337}, there exists a finite subset $J \subset \Nb$ with 
$$
Z \cap \Oc' = \Oc' \cap \bigcap_{j \in J} h_j^{-1}(0).
$$
Hence 
$$
Z \cap \Oc' = f^{-1}(0)
$$
where $f = \sum_{j \in J} h_j^2|_{\Oc'}$. Since $x \in Z$ was arbitrary, $Z$ is a real analytic subset of $\Uc$.
\end{proof}

We also have the following basic fact. 

\begin{proposition}\label{prop:ra set contains open set implies it is everything} Suppose $\Uc \subset \Rb^d$ is an open connected set and $Z \subset \Uc$ is real analytic. If $Z$ has non-empty interior, then $Z = \Uc$.
\end{proposition}

\begin{proof} Let $I$ denote the interior of $Z$. Since $\Uc$ is connected, it suffices to show that $I$ is closed in $\Uc$. 

Suppose $x$ is in the closure of $I$ in $\Uc$. Then $x \in Z$. So there exist an open neighborhood $\Oc$ of $x$ and a real analytic function $h : \Oc \rightarrow \Rb$ with 
$$
Z \cap \Oc = h^{-1}(0).
$$
If $x_n \rightarrow x$ where $x_n \in I \cap \Oc$, then the above implies that $h$ vanishes in a neighborhood of each $x_n$. Hence every partial derivative of $h$ vanishes at each $x_n$. Then, every partial derivative of $h$ vanishes at $x$ and so $h$ vanishes in a neighborhood of $x$. Hence $x \in I$.
\end{proof}


\section{Asymptotic behavior of automorphisms}

Given $g \in \Aut(\Bb^m) = \PU(m,1)$, let $\tilde g \in \Usf(m,1)$ denote some lift of $g$. Then let $\lambda_1(g)$ denote the largest absolute value of the eigenvalues of $\tilde g$, and let $\sigma_1(g)$ denote the largest singular value of  $\tilde g$. Notice that both of these quantities are independent of the choice of lift. Further, it follows from the spectral radius formula that 
\begin{equation}\label{eqn:spectral radius formula}
\lambda_1(g) =\lim_{n \rightarrow \infty} \sigma_1(g^n)^{1/n}.
\end{equation} 

We start by observing the following equivalence. 

\begin{theorem}\label{thm:lambda1 char loxo} If $g \in \Aut(\Bb^m)$, then $g$ is loxodromic if and only if $\lambda_1(g) > 1$. \end{theorem} 

\begin{proof} Notice that if $k \in \Msf$ and $t > 0$, then $\lambda_1(ka_t) = e^t>1$. So the forward direction follows from Theorem~\ref{thm:structure of loxodromic}. 

For the backward direction, it suffices to show that $\lambda_1(g) =1$ when $g$ is elliptic or parabolic. 

First, suppose that $g$ is elliptic. Since $\Aut(\Bb^m)$ acts transitively on $\Bb^m$, see ~\cite[Corollary 2.2.2]{Abate1989}, we can conjugate $g$ so that $g(\vect{0}) =\vect{0}$. Since $\Ksf$ is the stabilizer of $\vect{0}$ in $\Aut(\Bb^m)$,  we have $g \in \Ksf$ and $\lambda_1(g) =1$. 

Next, suppose that $g$ is parabolic. Since $\Aut(\Bb^m)$ acts transitively on $\partial \Bb^m$, we can assume that $g(-e_1)=-e_1$. Then ~\cite[Proposition 2.2.10]{Abate1989} implies that $g$ is conjugate in $\mathsf{PGL}(m+1, \Cb)$ to an upper triangular matrix with complex numbers of magnitude one on the diagonal. Hence $\lambda_1(g) =1$. 
\end{proof} 

The next two results relate $\lambda_1(g)$ to the dynamical behavior of the powers $\{g^n\}$. Recall from Theorem~\ref{thm:structure of loxodromic} that if $g \in \Aut(\Bb^m)$ is loxodromic, then we can label the fixed points of $g$ as $x_g^+, x_g^-$ so that 
$$
\lim_{n \rightarrow \pm \infty} g^{n}(z)= x_g^{\pm}
$$
for all $z \in \overline{\Bb^m} \smallsetminus \{x_g^{\mp} \}$. The next result computes the rate of convergence. 

\begin{theorem}\label{thm: contraction rates of loxodromics}
    Suppose $g \in \Aut(\Bb^m)$ is loxodromic and $L_g \subset \mathbb{C}^m$ is the complex affine line containing $x^\pm_g$.
    If $z \in \overline{\Bb^m} \smallsetminus  \{x^+_g, x^-_g\}$, then 
$$
\lim_{n \rightarrow \infty} \frac{1}{n} \log \|g^n(z) - x^+_g\| = 
\begin{cases} 
-2\log \lambda_1(g) & \text{if } z \in L_g \\
-\log \lambda_1(g) & \text{if } z \notin L_g 
\end{cases}. $$
\end{theorem}

\begin{proof} Recall from Theorem~\ref{thm:structure of loxodromic} that $g = ha_tkh^{-1}$ for some $h \in \mathsf{Aut}(\Bb^m)$, $k \in \mathsf{M}$, and $t > 0.$ Then $x^\pm_g = h(\pm e_1)$. Since $h$ acts smoothly on $\overline{\Bb^m}$, there exists some $C > 1$ such that
\begin{equation*}
    C^{-1}\|z_1 - z_2\|  \leq \|h(z_1) - h(z_2)\| \leq C\|z_1 - z_2\| 
\end{equation*}
for all $z_1,z_2 \in \overline{\Bb^m}$.
Thus 
\begin{align*}
\lim_{n \rightarrow \infty} \frac{1}{n} \log \|g^n(z) - x^+_g\| & =\lim_{n \rightarrow \infty} \frac{1}{n} \log \|h(a_tk)^nh^{-1}(z) - h(e_1)\|\\
& =\lim_{n \rightarrow \infty} \frac{1}{n} \log \|(a_{t}k)^n(h^{-1}(z)) - e_1\|.
\end{align*}
Since $h^{-1}$ maps $L_g$ to $\Cb \cdot e_1=L_{a_t k}$ and $\lambda_1(g) =\lambda_1(a_t k)$, it suffices to consider the case where $g = a_t k$ for some $t > 0$ and $k \in \mathsf{M}$. We further note that 
$$
\lambda_1(a_t k) = e^t .
$$
Fix $z \in \overline{\Bb^m} \smallsetminus \{e_{1}, -e_{1}\}$. 

\medskip

\underline{Case 1:} Assume $z \in \Cb \cdot e_1$. Then $z = ce_1$ where $c \in \overline{\Db}\smallsetminus\{1, -1\}$ (here $\Db$ denotes the unit disc in $\Cb$). Notice that the definition of $\Msf$ implies $k(z)=z$. So    \begin{align*}
        \lim_{n \to \infty} \frac{1}{n} \log \|g^n(z) - e_1\| &= \lim_{n \to \infty} \frac{1}{n} \log \left\|a_{nt}(ce_1) - e_1\right\| \\
        &= \lim_{n \to \infty} \frac{1}{n} \log \left\|\frac{c\cosh nt + \sinh nt}{c\sinh nt + \cosh nt}e_1 - e_1\right\| \\
        & = \lim_{n \to \infty} \frac{1}{n} \log \frac{ \abs{(c-1)(\cosh nt - \sinh nt)}}{\abs{c\sinh nt + \cosh nt}} \\
        &= -2t,
    \end{align*}
where the last equality uses the fact that $c \neq \pm 1$. 

\medskip

\underline{Case 2:} Assume $z \notin \Cb \cdot e_1$. Then $z = (c,z')$ for some $c \in \Db$ and $z' \in \overline{\Bb^{m-1}} \smallsetminus \{\vect{0}\}.$ Notice that the definition of $\Msf$ implies $k^n(z)=(c,U^nz')$ for some $U \in \Usf(m-1)$. So 
 \begin{align*}
        \lim_{n \to \infty} \frac{1}{n} \log \|g^nz - e_1\| &= \lim_{n \to \infty} \frac{1}{n} \log \left\|a_{nt}(ce_1 + U^nz') - e_1\right\| \\
        &= \lim_{n \to \infty} \frac{1}{n} \log \left\|\frac{(c\cosh nt + \sinh nt)e_1 + U^nz'}{c\sinh nt + \cosh nt} - e_1\right\| \\
        &= \lim_{n \to \infty} \frac{1}{n} \log \frac{ \sqrt{  \abs{(c-1)(\cosh nt - \sinh nt)}^2 + \norm{ U^n z'}^2} }{\abs{c\sinh nt + \cosh nt}} \\
        & = -t, 
    \end{align*}
    where the last equality uses the fact that $\abs{c} < 1$ and $\norm{ U^n z'} = \norm{z'} \in (0,1]$. 
\end{proof}

\begin{proposition}\label{prop:approaching the boundary}     If $g \in \mathsf{Aut}(\mathbb{B}^m),$ then
    $$ \|g(\vect{0})\| = \frac{\sigma_1(g)^2 - 1}{\sigma_1(g)^2 + 1}. $$
    Hence, 
    $$ \lim_{n \to \infty} \frac{1}{n}\log(1 - \|g^n(\vect{0})\|) = -2\log\left(\lambda_1(g)\right). $$
\end{proposition}

\begin{proof}     Using the $\mathsf{KAK}$-decomposition, see Theorem~\ref{thm:KAK}, we can write $g = k_1a_tk_2$ where $k_1,k_2 \in \Ksf$ and $t \geq 0.$ Then 
    $$
     \|g(\vect{0})\| = \|a_t(\vect{0})\| = \| (\tanh t, 0, \dots, 0)\|=\tanh t = \frac{e^t - e^{-t}}{e^t + e^{-t}} = \frac{e^{2t}-1}{e^{2t}+1}. 
    $$
    Since $\sigma_1(g) =\sigma_1(a_t) = e^t$, the first assertion follows. Now by Equation~\eqref{eqn:spectral radius formula},
    \begin{equation*}
        \begin{split}
    \lim_{n \to \infty} \frac{1}{n}\log(1 - \|g^n(\vect{0})\|) &= \lim_{n \to \infty} \frac{1}{n}\log\left(\frac{2}{\sigma_1(g^n)^2 + 1}\right) \\
    &= \lim_{n \to \infty} -\log\left(\sigma_1(g^n)^\frac{2}{n}\right) \\
    &= -2\log\left(\lambda_1(g)\right)
        \end{split}
    \end{equation*}
     which concludes the proof.
\end{proof}



\section{The Image of Lines Through Fixed Points}

Given a loxodromic element $g \in \Aut(\Bb^m)$, let $x_g^+, x_g^-$ be the labeling of the fixed points as in Theorem~\ref{thm:structure of loxodromic}, and let $L_g$ denote the complex affine line containing these fixed points. In this section, we show that if a proper holomorphic map extends to a H\"older continuous map on the boundary with sufficiently large H\"older exponent, then the complex affine lines associated to loxodromic elements in $\Aut(\Bb^m)$ are mapped into complex affine lines.

\begin{theorem}\label{thm:loxo lines are mapped to loxo lines} Suppose $f : \Bb^m\rightarrow \Bb^M$ is a proper holomorphic map that extends to an $\alpha$-H\"older continuous map $\overline{\Bb^m} \rightarrow \overline{\Bb^M}$ for some $\alpha > 1/2$.  If $(\phi,\psi) \in \Gsf_f$ and $\phi$ is loxodromic, then $\psi$ is loxodromic. Moreover, 
$$
f\left(L_{\phi} \cap \overline{\Bb^m}\right) \subset L_{\psi} \cap \overline{\Bb^M}.
$$
\end{theorem} 

The rest of the section is devoted to the proof of the theorem. Notice that replacing $f$ by $g f$ for some $g \in \Aut(\Bb^M)$ does not change the H\"older regularity of the extension (since $g$ acts smoothly on $\overline{\Bb^M}$). Further,  $(\phi,\psi) \in \Gsf_f$ if and only if $(\phi, g \psi g^{-1}) \in \Gsf_{gf}$. So by picking $g$ appropriately, we can assume that 
$$
f(\vect{0})=\vect{0}.
$$ 

Next fix $C > 1$ such that 
$$
\norm{f(z_1)-f(z_2)} \leq C\norm{z_1-z_2}^\alpha
$$
for all $z_1,z_2 \in \overline{\Bb^m}$. 

\begin{lemma}\label{lem:bound with one}
If $z \in \Bb^m$, then 
$$
1-\norm{f(z)} \leq C(1-\norm{z})^{\alpha}.
$$
\end{lemma}

\begin{proof}
Since $C > 1$, this holds for $z=\vect{0}$. Suppose that $z \neq \vect{0}$. 

Recall that $f$ is proper and extends to a continuous map $\overline{\Bb^m} \rightarrow \overline{\Bb^M}$. Hence $f$ maps $\partial \Bb^m$ into $\partial \Bb^M$ and in particular $\norm{f(z/\norm{z})} = 1$. Thus,  
\begin{equation*}
1 - \norm{f(z)} \leq \norm{f \left( \frac{z}{\norm{z}} \right) - f(z)} \leq C \norm{\frac{z}{\norm{z}} - z}^\alpha =C (1 - \norm{z})^\alpha. \qedhere
\end{equation*} 
\end{proof}

\begin{proof}[Proof of Theorem~\ref{thm:loxo lines are mapped to loxo lines}] Suppose $(\phi,\psi) \in \Gsf_f$ and $\phi$ is loxodromic.

We first show that $\psi$ is loxodromic. By Proposition~\ref{prop:approaching the boundary}, 
$$
-2\log \lambda_1( \psi) = \lim_{n \rightarrow \infty}\frac{1}{n} \log\left(1 - \norm{\psi^n(\vect{0})}\right)
$$
and 
$$
-2\log \lambda_1( \phi) = \lim_{n \rightarrow \infty} \frac{1}{n}\log\left(1 - \norm{\phi^n(\vect{0})}\right).
$$
Since $\psi^n(\vect{0})=\psi^n(f(\vect{0})) = f(\phi^n (\vect{0}))$, Lemma~\ref{lem:bound with one} implies that 
\begin{align*}
-2\log \lambda_1( \psi) & = \lim_{n \rightarrow \infty} \frac{1}{n} \log\left(1 - \norm{\psi^n(f(\vect{0}))}\right) =  \lim_{n \rightarrow \infty}  \frac{1}{n}\log\left(1 - \norm{f(\phi^n(\vect{0}))}\right) \\
& \leq \alpha   \lim_{n \rightarrow \infty} \frac{1}{n} \log\left(1 - \norm{\phi^n(\vect{0})}\right) = -2\alpha \log \lambda_1( \phi). 
\end{align*}
So $\lambda_1(\psi) \geq \lambda_1(\phi)^{\alpha}$. Then Theorem~\ref{thm:lambda1 char loxo} implies that $\psi$ is loxodromic. 

Next we show that $f(L_{\phi} \cap \overline{\Bb^m})$ is a subset of $L_{\psi} \cap \overline{\Bb^M}$. Fix $z \in L_{\phi} \cap \overline{\Bb^m}$ and suppose for a contradiction that $f(z) \notin L_{\psi} \cap \overline{\Bb^M}$. By Theorem~\ref{thm: contraction rates of loxodromics}, we have
$$
\lim_{n \rightarrow \infty} \frac{1}{n} \log \|\phi^n(z) - x^+_\phi\| = -2 \log \lambda_1(\phi)
$$
and 
$$
\lim_{n \rightarrow \infty} \frac{1}{n} \log \|\psi^n(f(z)) - x^+_\psi\| = - \log \lambda_1(\psi). 
$$
The dynamical behavior of loxodromic elements (see Theorem~\ref{thm:structure of loxodromic}) and the fact that 
$$
\psi^n f = f \phi^n \quad \text{ for all $n \geq 1$}
$$
implies that $f( x_\phi^\pm) = x_\psi^\pm$. Since $f$ is $\alpha$-H\"older continuous, we then have 
\begin{align*}
 - \log \lambda_1(\psi) & = \lim_{n \rightarrow \infty} \frac{1}{n} \log \|\psi^n(f(z)) - x^+_\psi\| =  \lim_{n \rightarrow \infty} \frac{1}{n} \log \|f(\phi^n(z)) - f(x^+_\phi)\| \\
 & \leq \alpha  \lim_{n \rightarrow \infty} \frac{1}{n} \log \|\phi^n(z) - x^+_\phi\| = -2\alpha  \log \lambda_1(\phi).
 \end{align*} 
 Since $\alpha > 1/2$, this implies that $\lambda_1(\psi) >  \lambda_1(\phi)$. However, since $f(\vect{0})=\vect{0}$, the Schwarz lemma implies that $\norm{f(z)} \leq \norm{z}$ for all $z \in \Bb^m$. Then Proposition~\ref{prop:approaching the boundary} yields
\begin{align*}
-2\log \lambda_1( \psi) & = \lim_{n \rightarrow \infty} \frac{1}{n} \log\left(1 - \norm{\psi^n(f(\vect{0}))}\right) =  \lim_{n \rightarrow \infty}  \frac{1}{n}\log\left(1 - \norm{f(\phi^n(\vect{0}))}\right) \\
& \geq    \lim_{n \rightarrow \infty} \frac{1}{n} \log\left(1 - \norm{\phi^n(\vect{0})}\right) = -2 \log \lambda_1( \phi). 
\end{align*} 
Hence $\lambda_{1}(\psi) \leq \lambda_{1}(\phi)$, which is a contradiction.

It follows that $f(z) \in L_{\psi} \cap \overline{\Bb^M}$. Since  $z \in L_{\phi} \cap \overline{\Bb^m}$ was arbitrary, we have 
\begin{equation*}
f(L_{\phi} \cap \overline{\Bb^m}) \subset L_{\psi} \cap \overline{\Bb^M},
\end{equation*} as desired.
\end{proof} 


\section{Affine lines intersecting the limit set}


Recall that the limit set of a subgroup $\Gamma < \Aut(\Bb^m)$ is
$$
\Lambda(\Gamma) :=\{ x \in \partial \Bb^m : \text{ there exists $\{g_n\} \subset \Gamma$ with $g_n(\vect{0}) \rightarrow x$}\}. 
$$

In this section we derive the following corollary to Theorem~\ref{thm:loxo lines are mapped to loxo lines}. 

\begin{corollary}\label{cor:complex lines intersecting the limit set} Suppose $f : \Bb^m\rightarrow \Bb^M$ is a proper holomorphic map that extends to an $\alpha$-H\"older continuous map $\overline{\Bb^m} \rightarrow \overline{\Bb^M}$ for some $\alpha > 1/2$. Let $\Gamma$ denote the image of the projection $\Gsf_f \rightarrow \Aut(\Bb^m)$. If $\Gamma$ is Zariski dense and $L$ is a complex affine line intersecting $\Lambda(\Gamma)$ at two or more points, then 
$$
f(L \cap \overline{\Bb^m}) 
$$
is contained in some complex affine line in $\Cb^M$. 
\end{corollary} 

Corollary~\ref{cor:complex lines intersecting the limit set}  is a consequence of Theorem~\ref{thm:loxo lines are mapped to loxo lines} and the following well-known result. 

\begin{proposition}\label{prop:density of loxs} Suppose $\Gamma< \Aut(\Bb^m)$ is a Zariski dense subgroup. If $x,y \in \Lambda(\Gamma) $ are distinct, then there exist loxodromic elements $\{\gamma_n\} \subset \Gamma$ where $x^+_{\gamma_n} \rightarrow x$ and $x^-_{\gamma_n} \rightarrow y$.  \end{proposition}

Since the proof is short we include it. 

\begin{proof}[Proof of Proposition~\ref{prop:density of loxs}] Fix $\{g_n\}, \{h_n\} \subset \Gamma$ with $g_n(\vect{0}) \rightarrow x$ and $h_n(\vect{0}) \rightarrow y$. Passing to subsequences we can suppose that $g_n^{-1}(\vect{0}) \rightarrow x^-$ and $h_n^{-1}(\vect{0}) \rightarrow y^-$. 

We first consider the case where $y^- \neq x^-$. Let $\gamma_n:= g_n h_n^{-1}$. Then Theorem~\ref{thm:North South}  implies that 
\begin{equation}\label{eqn:gamma n behavior}
\lim_{n \rightarrow \infty} \gamma_n(\vect{0}) = x \quad \text{and} \quad \lim_{n \rightarrow \infty} \gamma_n^{-1}(\vect{0}) = y.
\end{equation} 
\
Fix a compact neighborhood $B \subset \partial \Bb^m$ of $x$ such that $B$ is homeomorphic to a closed ball in $\Rb^{2m-1}$ and $y \notin B$. Since $x \neq y$, Theorem~\ref{thm:North South} implies that $\gamma_n(B) \subset B$ for $n$ sufficiently large. Then for such $n$, $\gamma_n$ has a fixed point $x_n^+$ in $B$. Since $B$ can be chosen arbitrarily small, we must have $x_n^+ \rightarrow x$. Likewise, for $n$ sufficiently large $\gamma_n^{-1}$ has a fixed point $x_n^-$  and $x_n^- \rightarrow y$. Moreover, $x_n^+\neq x_n^-$ for large $n$.

To finish the proof we need to show that $\gamma_n$ is loxodromic for $n$ sufficiently large. Suppose not. Then after passing to a subsequence we can suppose that each $\gamma_n$ is non-loxodromic and that $x_n^+$, $x_n^-$ exist for all $n$. Since $\Aut(\Bb^m)$ acts doubly transitively on $\partial\Bb^m$, see~\cite[Corollary 2.2.5]{Abate1989}, for each $n$ we can fix $f_n \in \Aut(\Bb^m)$ such that $f_n(x_n^\pm) = \pm e_1$. Then $f_n \gamma_n f_n^{-1}$ fixes $\pm e_1$ and it is easy to show that $f_n \gamma_n f_n^{-1} = k_n a_{t_n}$ for some $k_n \in \Msf$ and $t_n \in \Rb$. Since $\gamma_n$ is not loxodromic, we must have $t_n =0$. Then $f_n\gamma_nf_n^{-1}$ fixes the set $\Cb \cdot e_1 \cap \Bb^m$ pointwise. This implies that $\gamma_n$ fixes the set $L_n \cap \Bb^m$ pointwise, where $L_n$ is the complex affine line containing $x_n^+, x_n^-$. Since  $x_n^+ \rightarrow x$,  $x_n^- \rightarrow y$, and $x \neq y$ we can fix $z \in \Bb^m$ and $z_n \in L_n \cap \Bb^m$ such that $z_n \rightarrow z$. Then 
$$
z =  \lim_{n \rightarrow \infty} z_n=\lim_{n \rightarrow \infty} \gamma_n(z_n).
$$
However, Theorem~\ref{thm:North South} and Equation~\eqref{eqn:gamma n behavior} imply that 
$$
x = \lim_{n \rightarrow \infty} \gamma_n(z_n).
$$
So we have a contradiction. This completes the proof when $y^- \neq x^-$. 

Next consider the case where $y^- = x^-$. Using Proposition~\ref{prop:properties Z dense}, we can fix $\gamma \in \Gamma$ with $\gamma y^- \neq x^-$. Then let $\gamma_n : = g_n \gamma h_n^{-1}$ and repeat the argument above. 
\end{proof} 

\begin{proof}[Proof of Corollary~\ref{cor:complex lines intersecting the limit set}] Suppose $L$ is a complex affine line intersecting $\Lambda(\Gamma)$ at two or more points. Fix $x,y \in L \cap \Lambda(\Gamma)$ distinct. Proposition~\ref{prop:density of loxs} implies that there exist loxodromic elements $\{\gamma_n\} \subset \Gamma$ where $x^+_{\gamma_n} \rightarrow x$ and $x^-_{\gamma_n} \rightarrow y$. By Theorem~\ref{thm:loxo lines are mapped to loxo lines}, each image
$$
f\left(L_{\gamma_n} \cap \overline{\Bb^m}\right)
$$
is contained in a complex affine line. 

Suppose for a contradiction that $f(L \cap \overline{\Bb^m})$ is not contained in a complex affine line. Then there exist $u,v,w \in L \cap \overline{\Bb^m}$ such that $f(u),f(v),f(w)$ are not contained in a complex affine line. Then for each $n \geq 1$, we can fix $u_n,v_n,w_n \in L_{\gamma_n} \cap \overline{\Bb^m}$ such that $u_n \rightarrow u$, $v_n \rightarrow v$, and $w_n \rightarrow w$. For each $n \geq 1$, the points $f(u_n), f(v_n), f(w_n)$ are contained in a complex affine line. Hence, by the continuity of $f$ on $\overline{\Bb^m}$, the points $f(u), f(v), f(w)$ are contained in a complex affine line. Thus, we have a contradiction.
\end{proof}

\section{Analytic sets}\label{sec:analytic sets}

Suppose $f : \Bb^m\rightarrow \Bb^M$ is a proper holomorphic map which extends continuously to $\overline{\Bb^m}$. For $x \in \partial \Bb^m$, let
\[
Z_{x} := \left\{y \in \partial\Bb^m\smallsetminus\{x\} : f\left(L_{xy} \cap \overline{\Bb^m}\right) \text{ is contained in a complex affine line}\right\},
\]
where $L_{xy}$ is the complex affine line containing $x,y$.

The goal of this section is to show that \(Z_{x}\) is a real analytic subset.

\begin{theorem}\label{thm: Z is a real analaytic variety}
  If \(x \in \partial \Bb^m\), then \(Z_x\) is a real analytic subset in $\Cb^m\smallsetminus \{x\}$. 
\end{theorem}

The rest of the section is devoted to the proof of the theorem. Fix $x \in \partial \Bb^m$. We first verify that $Z_x$ is relatively closed.

\begin{lemma}
 $Z_x$ is relatively closed in $\Cb^m\smallsetminus \{x\}$.
\end{lemma}

\begin{proof} We argue exactly as in the proof of Corollary~\ref{cor:complex lines intersecting the limit set}. Suppose $\{y_n\} \subset Z_x$ converges to $y \in \Cb^m\smallsetminus \{x\}$.  Suppose for a contradiction that $y \notin Z_x$. Then there exist $u,v,w \in L_{xy} \cap \overline{\Bb^m}$ such that $f(u),f(v),f(w)$ are not contained in a complex affine line.  Then for each $n \geq 1$, we can fix $u_n,v_n,w_n \in L_{xy_n} \cap \overline{\Bb^m}$ such that $u_n \rightarrow u$, $v_n \rightarrow v$, and $w_n \rightarrow w$. For each $n \geq 1$, the points $f(u_n), f(v_n), f(w_n)$ are contained in a complex affine line. Hence, by the continuity of $f$ on $\overline{\Bb^m}$, the points $f(u), f(v), f(w)$ are contained in a complex affine line. Thus, we have a contradiction.
\end{proof}

Using Theorem~\ref{thm:noetherian property}, it suffices to fix $y_0 \in Z_x$, then find an open neighborhood $\Uc$ of $y_0$ and countably many real analytic functions $\{ h_n : \Uc \rightarrow \Rb\}_{n=0}^\infty$ such that 
$$
Z_x \cap \Uc = \bigcap_{n =0}^\infty h_n^{-1}(0).
$$
Fix an open neighborhood $\Uc$ of $y_0$ in $\Cb^m\smallsetminus \{x\}$ such that 
$$
\left\{  \frac{x+y}{2} \right\} \bigcup \left\{ \frac{2n+1}{4n} y + \frac{2n-1}{4n} x  : n\geq 1\right\} \subset \Bb^m.
$$
for all $y \in \Uc$. For $y \in \Uc$, set
$$
z(y) := \frac{x+y}{2}.
$$
Fix $u_2, \dots, u_M$ such that 
$$
f\left( z(y_0) \right) - f(x), u_2, \dots, u_M
$$
forms a basis of $\Cb^M$. Next, for $y \in \Uc$,  consider the $M$-by-$M$ matrix
$$
A(y) := \left[ f\left(z(y) \right) - f(x), u_2, \dots, u_M\right].
$$
Notice that $y\in \Uc \mapsto \det A(y) \in \Cb$ is continuous, so after shrinking $\Uc$ we can assume that 
$$
\det A(y) \neq 0 
$$
for all $y \in \Uc$. Lastly, let $\pi : \Cb^{M} \rightarrow \Cb^{M-1}$ be the projection given by
\begin{align*}
\pi(z_1,\dots, z_M) = (z_2, \dots, z_M).
\end{align*}

\begin{lemma}\label{lem:detector gadget}  If $y \in \Uc$ and $w \in \Bb^m$, then 
$$
f(w) \in L_{f(x)f(z(y))} \Longleftrightarrow \pi \big[A(y)^{-1} (f(w)-f(x)) \big] =\vect{0}.
$$
\end{lemma} 

\begin{proof} Notice
\begin{align*} 
f(w) \in L_{f(x)f(z(y))}&  \Longleftrightarrow f(w) - f(x) \in \Cb \cdot (f(z(y)) - f(x)) \\
& \Longleftrightarrow  A(y)^{-1} (f(w)-f(x)) \in \Cb \cdot e_1\\
& \Longleftrightarrow \pi \big[A(y)^{-1} (f(w)-f(x))\big] =\vect{0}. \qedhere
\end{align*} 
\end{proof}

Set $h_0(z) = 1-\norm{z}^2$ and for $n \in \Nb$ define $h_n : \Uc \rightarrow \Rb$ by 
$$
h_n(y) = \norm{ \pi \bigg[A(y)^{-1} \left(f\left(  \frac{2n+1}{4n} y + \frac{2n-1}{4n} x \right)-f(x)\right) \bigg]}^2.
$$
Since $\frac{2n+1}{4n} y + \frac{2n-1}{4n} x \in \Bb^m$, each $h_n$ is real analytic. 
\begin{lemma} We have
$$
Z_x \cap \Uc = \bigcap_{n=0}^\infty h_n^{-1}(0).
$$
\end{lemma} 

\begin{proof} First assume $y \in Z_x \cap \Uc$. Then $y \in \partial \Bb^m$ and so $h_0(y)=0$. Further,
$$
f(w) \in  L_{f(x)f(z(y))} 
$$
for all $w \in L_{xy} \cap \Bb^m$. So Lemma~\ref{lem:detector gadget} implies that $h_n(y) = 0$ for all $n \geq 1$.

Next assume that $y \in \bigcap_{n} h_n^{-1}(0)$. Let 
$$
\Oc : = \{ \lambda \in \Cb : \lambda y + (1-\lambda )x \in \Bb^m\}
$$
and consider the holomorphic map $h : \Oc \rightarrow \Cb^{M-1}$ defined by 
$$
h(\lambda) = \pi \big[ A(y)^{-1} \left(f\left(  \lambda y + (1-\lambda )x \right)-f(x)\right) \big].
$$
By assumption, $h$ vanishes at each $\frac{2n+1}{4n}$. Since $\frac{2n+1}{4n} \rightarrow \frac{1}{2}$, this implies that $h$ is identically zero. Then Lemma~\ref{lem:detector gadget} implies that 
$$
f(w) \in  L_{f(x)f(z(y))} 
$$
for all $w \in L_{xy} \cap \Bb^m$. So $y \in Z_x \cap \Uc$. 
\end{proof} 

\begin{remark} Although this is not needed for the proof of Theorem~\ref{thm:main}, we note that it is possible to show  that the set 
\[
\left\{z \in\Bb^m : f\left(L_{xz} \cap \overline{\Bb^m}\right) \text{ is contained in a complex affine line}\right\}.
\]
is a complex subvariety of $\Bb^m$. 
\end{remark} 

\section{Analytic envelopes}

In this section we show that the limit set of a Zariski dense subgroup cannot be contained in a proper real analytic subset of $ \partial \Bb^m \smallsetminus \{x\}$.

\begin{theorem}\label{thm:analytic envelopes} Suppose $\Gamma < \Aut(\Bb^m)$ is Zariski dense. If $x \in \partial \Bb^m$ and $Z \subset \Cb^m \smallsetminus \{ x\}$ is a real analytic subset with 
$$
\Lambda(\Gamma)\smallsetminus \{x\} \subset Z \subset \partial \Bb^m \smallsetminus \{x\},
$$ 
then $Z = \partial \Bb^m \smallsetminus \{ x\}$.
\end{theorem}
The argument we provide is similar to the proof of ~\cite[Lemma 2.11]{MR4589560}. For the proof it will be more convenient to work in the Siegel upper half-space model of the unit ball and so we briefly discuss this model before starting the proof. 

\subsection{The Siegel upper half-space model} 

Let 
$$
\Pc^m := \left\{ z \in \Cb^m : {\rm Im}(z_1) > \abs{z_2}^2 + \cdots + \abs{z_m}^2 \right\}.
$$
Then the map 
$$
F(z) = \left(  i\frac{1-z_1}{1+z_1}, \frac{iz_2}{1+z_1}, \dots, \frac{iz_m}{1+z_1} \right)
$$
is a biholomorphism from $\Bb^m $ to $ \Pc^m$ with inverse 
$$
F^{-1}(z) = \left( \frac{i-z_1}{z_1+i}, \frac{2z_2}{z_1+i}, \dots, \frac{2z_m}{z_1+i} \right).
$$
Notice that $F$ extends to a real analytic diffeomorphism on the boundary, which we will also denote by $F$,
$$
F: \partial \Bb^m \smallsetminus \{-e_1\} \rightarrow \partial \Pc^m.
$$
Further, $F(e_1)=\vect{0}$ and 
$$
\lim_{z \rightarrow -e_1} \norm{F(z)}=\infty.
$$

In what follows we write $z \in \Cb^m$ as $z=(v,w)$ where $v \in \Cb$ and $w \in \Cb^{m-1}$. Then 
\begin{equation}\label{eqn:at in hyperboloid} 
F \circ a_t \circ F^{-1}(v,w) = (e^{-2t}v, e^{-t}w)
\end{equation}
for all $t \in \Rb$. Further, if $k \in \Msf$, then 
\begin{equation}\label{eqn:M in hyperboloid} 
F \circ k \circ F^{-1}(v,w) = (v, Uw)
\end{equation} 
for some $U \in \Usf(m-1)$.

\subsection{Proof of Theorem~\ref{thm:analytic envelopes}}

Since $\Lambda(\Gamma)$ is infinite (see Proposition~\ref{prop:properties Z dense}), using  Proposition~\ref{prop:density of loxs}  we can fix a loxodromic element $\phi_0 \in \Gamma$ with fixed points $x^\pm_{\phi_0} \in \partial \Bb^m \smallsetminus \{ x\}$. Let $Z$ be as in the statement of Theorem~\ref{thm:analytic envelopes}. By conjugating $\Gamma$ and transporting $Z$ and $x$ by the same automorphism, we can assume that $\phi_0 = k a_t$ for some $t > 0$ and $k \in \Msf$, and $x^\pm_{\phi_0} = \pm e_1$. 

Fix a connected neighborhood $\Oc_0$ of $e_1$ in $\partial \Bb^m \smallsetminus \{x, -e_1\}$ such that 
$$
\Oc_0 \cap Z = h_0^{-1}(0)
$$
for some real analytic function $h_0 : \Oc_0 \rightarrow \Rb$. 

If $h_0$ is constant, then $\Oc_0 \cap Z = \Oc_0$ and hence $Z$ contains an open subset of $\partial \Bb^m \smallsetminus \{x\}$. By stereographic projection, there is a real analytic diffeomorphism from $\partial \Bb^m \smallsetminus \{x\}$ to $\Rb^{2m-1}$ and the image of $Z$ under this map is a real analytic subset of $\Rb^{2m-1}$. Then Proposition~\ref{prop:ra set contains open set implies it is everything} implies that $Z =\partial\Bb^m \smallsetminus \{x\}$. So, it suffices to assume that $h_0$ is non-constant and obtain a contradiction.

Let $F : \Bb^m \rightarrow \Pc^m$ be the biholomorphism introduced in the previous subsection. 
We also have a real analytic diffeomorphism $\pi: \partial \Pc^m \rightarrow \Rb \times \Cb^{m-1}$ given by 
$$
\pi(z_1,\dots, z_m) = ({\rm Re}(z_1), z_2, \dots, z_m). 
$$
Let $\zeta : \partial \Bb^m \smallsetminus \{-e_1\} \rightarrow \Rb \times \Cb^{m-1}$ be the real analytic diffeomorphism given by $\zeta(z) = \pi \circ F(z)$. For the rest of the section, let 
$$
\phi : =\zeta\phi_0 \zeta^{-1} , \quad  \Lambda : = \zeta(\Lambda(\Gamma) \smallsetminus \{-e_1\}), \quad \Oc:= \zeta(\Oc_0), \quad \text{and} \quad h : = h_0\circ \zeta^{-1}.
$$
We summarize the basic properties of these objects in the next lemma. 

\begin{lemma}\label{lem:basic properties of translated objects} \
\begin{enumerate}
\item $h : \Oc \rightarrow \Rb$ is a non-constant real analytic function.
\item $\Oc$ is a neighborhood of $\vect{0}$ in $\Rb \times \Cb^{m-1}$ and 
$$
\Lambda \cap \Oc \subset h^{-1}(0).
$$
\item If $P : \Rb \times \Cb^{m-1} \rightarrow \Rb$ is a real polynomial and $P|_{\Lambda} \equiv 0$, then $P \equiv 0$.
\item There exist $t > 0$ and $U \in \Usf(m-1)$ such that 
$$
\phi(v,w) = (e^{-2t}v, e^{-t}Uw)
$$
for all $(v,w) \in \Rb \times \Cb^{m-1}$.
\end{enumerate}
\end{lemma}

\begin{proof} Parts (1) and (2) are immediate from the definitions.  Since $\phi = ka_t$ for some $t > 0$ and $k \in \Msf$, part (4) is a consequence of Equations~\eqref{eqn:at in hyperboloid} and \eqref{eqn:M in hyperboloid}.

For Part (3), notice that $\zeta$ is the restriction of a real rational function on $\Cb^m$. So we can write $P\circ \zeta = \frac{P_1}{P_2}$ where $P_1,P_2$ are real polynomials and $P_2$ is non-vanishing on $\partial \Bb^m \smallsetminus \{-e_1\}$. Thus $P_1|_{ \Lambda(\Gamma) \smallsetminus \{-e_1\}}\equiv 0$ and so by continuity $P_1|_{ \Lambda(\Gamma)}\equiv 0$. Then Proposition~\ref{prop:properties Z dense} implies that $P_1|_{\partial \Bb^m} \equiv 0$. So $P \equiv 0$.
\end{proof}

For $w \in \Cb^{m-1}$ and $\beta \in \Zb_{\geq 0}^{2m-2}$ let 
$$
w^\beta := {\rm Re}(w_1)^{\beta_1}{\rm Im}(w_1)^{\beta_2}\cdots {\rm Re}(w_{m-1})^{\beta_{2m-3}}{\rm Im}(w_{m-1})^{\beta_{2m-2}}.
$$
Also let 
$$
\abs{\beta} = \beta_1+\cdots + \beta_{2m-2}.
$$
Since $h$ is real analytic, by shrinking $\Oc_0$, and hence also $\Oc$, we can assume that 
$$
h(v,w) = \sum_{\alpha \in \Zb_{\geq 0}}\sum_{\beta \in \Zb^{2m-2}} c_{\alpha, \beta} v^\alpha w^\beta
$$
on $\Oc$, where the coefficients $\{c_{\alpha,\beta}\}$ are real.

Since $h$ is non-constant, at least one coefficient is non-zero. Let 
$$
N : = \min\left\{ 2\alpha + \abs{\beta} : c_{\alpha,\beta} \neq 0\right\} \in \Zb_{\geq 0}.  
$$
Then using Taylor's theorem and shrinking $\Oc_0$, we can write 
$$
h(v,w) = \sum_{\alpha +\abs{\beta} \leq N}c_{\alpha, \beta} v^\alpha w^\beta + E(v,w)
$$
and there exists $C > 0$ such that  $E$ satisfies 
\begin{equation}
    \abs{E(v,w)}\leq C \norm{(v,w)}^{N+1}
    \label{taylor_approx}
\end{equation}
on $\Oc$. 

Let 
$$
P(v,w) := \sum_{2\alpha +\abs{\beta} = N}c_{\alpha, \beta} v^\alpha w^\beta,
$$
which is a non-constant real polynomial. 

Recall that 
$$
\phi(v,w) = (e^{-2t}v, e^{-t}Uw)
$$
where $t > 0$ and $U \in \Usf(m-1)$. Pick $n_j \rightarrow \infty$ such that $U^{n_j} \rightarrow \Id$.

\begin{lemma} If $(v,w) \in \Rb \times \Cb^{m-1}$, then $ h\phi^{n_j}(v,w)$ is well-defined for $j$ sufficiently large and 
$$
\lim_{j \rightarrow \infty}e^{N n_j t} h\phi^{n_j}(v,w) = P(v,w).
$$
Moreover, $P|_{\Lambda} \equiv 0$.
\end{lemma}

\begin{proof} Fix $(v,w) \in \Rb \times \Cb^{m-1}$. Since $\lim_{j \rightarrow \infty} \phi^{n_j}(v,w) = \vect{0}$, for $j$ sufficiently large $\phi^{n_j}(v,w) \in \mathcal{O}$ and hence $h\phi^{n_j}(v,w)$ is well-defined for $j$ sufficiently large.

By Equation~\eqref{taylor_approx}, for $j$ sufficiently large
    \begin{align*}
       \abs{e^{N n_jt}E(\phi^{n_j}(v,w))} &\leq Ce^{N n_jt}\norm{(e^{-2n_jt}v,e^{-n_jt}U^{n_j}w)}^{N + 1} \\
        &= Ce^{N n_jt - (N + 1)n_jt}\left(\sqrt{\|e^{-n_jt}v\|^2 + \|w\|^2}\right)^{N + 1} \\
        &\leq C\|(v, w)\|^{N+ 1}e^{-n_jt}.
    \end{align*}
    Hence $\lim_{j \rightarrow \infty} \abs{e^{N n_jt}E(\phi^{n_j}(v,w))}=0$. 
    
    Note that 
    $$
     \abs{(U^{n_j}w)^\beta} \leq \norm{U^{n_j}w}^{|\beta|}_\infty \leq \|U^{n_j}w\|^{|\beta|} = \|w\|^{|\beta|}.
     $$
Hence if $2\alpha +\abs{\beta} > N$, then 
    \begin{align*}
     \lim_{j \rightarrow \infty}   \abs{c_{\alpha, \beta}e^{-n_jt(2\alpha + |\beta| - N)}v^\alpha(U^{n_j}w)^\beta} \leq\lim_{j \rightarrow \infty} \abs{c_{\alpha,\beta}}\abs{v}^\alpha  \|w\|^{|\beta|}  e^{-n_jt} =0. 
    \end{align*}
    
  Recall that   $U^{n_j} \to \Id$. So if $2\alpha +\abs{\beta} = N$, then 
    $$ 
    \lim_{j \to \infty} c_{\alpha, \beta}e^{-n_jt(2\alpha + |\beta| - N)}v^\alpha(U^{n_j}w)^\beta = c_{\alpha,\beta}v^\alpha w^\beta. 
    $$

        Therefore, \begin{align*}
       \lim_{j \rightarrow \infty} e^{N n_jt} h\phi^{n_j}(v,w) &= \lim_{j \rightarrow \infty} \sum_{\alpha +\abs{\beta} \leq N} c_{\alpha, \beta}e^{-n_jt(2\alpha + |\beta| - N)}v^\alpha(U^{n_j}w)^\beta + e^{N n_jt}E(\phi^{n_j}(v,w)) \\
       & = P(v,w).
    \end{align*}

    Now suppose that $(v,w) \in \Lambda$. Since  $\Lambda(\Gamma)$ is $\Gamma$-invariant and $\phi^{n_j}(v,w) \rightarrow \vect{0}$, for $j$ sufficiently large $\phi^{n_j}(v,w) \in \Lambda \cap \Oc$. Thus,
    $$
    P(v,w) = \lim_{j \rightarrow \infty} e^{N n_jt} h\phi^{n_j}(v,w) =0
    $$
    by Lemma~\ref{lem:basic properties of translated objects}  part (2).  Hence $P|_\Lambda \equiv 0$.
\end{proof} 

Now $P$ is a non-constant real polynomial and $P|_\Lambda \equiv 0$, so we have a contradiction with Lemma~\ref{lem:basic properties of translated objects}  part (3). This concludes the proof of Theorem \ref{thm:analytic envelopes}.

\section{Finishing the proof}\label{sec:final section} 

In this section we complete the proof of Theorem~\ref{thm:main}, which we restate here.

\begin{theorem} If $2 \leq m < M$ and $f: \Bb^m \rightarrow \Bb^M$ is a proper holomorphic map where 
\begin{enumerate} 
\item the image of the projection $\Gsf_f \rightarrow \Aut(\Bb^m)$ is Zariski dense,
\item $f$ extends to an $\alpha$-H\"older continuous map $\overline{\Bb^m} \rightarrow \overline{\Bb^M}$ for some $\alpha > \frac{1}{2}$, 
\end{enumerate} 
then there exist $\phi_1 \in \Aut(\Bb^M)$ and $\phi_2 \in \Aut(\Bb^m)$ such that  
\begin{align*}
\phi_1 \circ f \circ \phi_2(z) = (z,\vect{0})
\end{align*}
for all $z \in \Bb^m$. 
\end{theorem}

\begin{proof} Arguing as in the start of the proof of Theorem~\ref{thm:loxo lines are mapped to loxo lines}, we can replace $f$ by $gf$ for some $g \in \Aut(\Bb^M)$ and assume that 
$$
f(\vect{0})=\vect{0}.
$$
Next, by replacing $f$ by $kf$ for some $k \in \mathsf{U}(M)$ we can assume that the linear span of $f(\Bb^m)$ is $\Cb^{M'} \times \{0\}$ for some $M' \leq M$. Then 
$$
f(z) = (f_0(z), \vect{0})
$$
where $f_0 : \Bb^m \rightarrow \Bb^{M'}$ is a proper holomorphic map. We claim that $f_0$ satisfies the hypothesis of the theorem. Clearly, $f_0$ extends to an $\alpha$-H\"older continuous map $\overline{\Bb^m} \rightarrow \overline{\Bb^{M'}}$. We verify the other hypothesis in the next claim.

\medskip

\noindent \textbf{Claim:}  The image of the projection $\Gsf_{f_0} \rightarrow \Aut(\Bb^m)$ is Zariski dense.

\medskip

\noindent \emph{Proof of Claim:} Since the image of the projection $\Gsf_{f} \rightarrow \Aut(\Bb^m)$ is Zariski dense, it suffices to fix $(\phi, \varphi) \in \Gsf_f$ and show that $(\phi, \varphi_0)\in \Gsf_{f_0}$ for some $\varphi_0 \in \Aut(\Bb^{M'})$. Notice that 
$$
f(\Bb^m) = f \phi(\Bb^m) = \varphi f(\Bb^m).  
$$
Then, since $\varphi$ is a fractional linear transformation and $\vect{0} \in f(\Bb^m)$, the map $\varphi$ preserves the linear span of $f(\Bb^m)$ intersected with $\Bb^M$ (see~\cite[Lemma 2.2.6]{Abate1989}), that is 
$$
\varphi\left( \Bb^{M'} \times \{\vect{0}\}\right) = \Bb^{M'} \times \{\vect{0}\}.
$$
So there exists $\varphi_0 \in \Aut(\Bb^{M'})$ such that 
$$
\varphi(z,\vect{0}) = (\varphi_0(z),\vect{0})
$$
for all $z \in \Bb^{M'}.$ Then $(\phi, \varphi_0)\in \Gsf_{f_0}$, which completes the proof of the claim.

\hfill $\blacktriangleleft$

\medskip

Thus $f_0$ satisfies the hypotheses of the Theorem. To avoid cumbersome notation, we now replace $f$ with $f_0$ and assume that $f(\Bb^m)$ spans $\Cb^{M}$. Since $f$ is proper, we must have $M \geq m$. 

As in Section~\ref{sec:analytic sets}, for $x \in \partial \Bb^m$ let 
\[
Z_{x} := \left\{y \in \partial\Bb^m\smallsetminus\{x\} : f\left(L_{xy} \cap \overline{\Bb^m}\right) \text{ is contained in a complex affine line}\right\}.
\]
By Theorem~\ref{thm: Z is a real analaytic variety}, each $Z_x$ is a real analytic subset in $\Cb^m\smallsetminus \{x\}$. 

\medskip

\noindent \textbf{Claim:}  $Z_x = \partial \Bb^m\smallsetminus \{x\}$ for all $x \in \partial \Bb^m$. 

\medskip

\noindent \emph{Proof of Claim:} First assume that $x \in \Lambda(\Gamma)$. Then Corollary~\ref{cor:complex lines intersecting the limit set} implies that 
$$
\Lambda(\Gamma) \smallsetminus \{x\} \subset Z_x. 
$$
Then Theorem~\ref{thm:analytic envelopes} implies that $Z_x = \partial \Bb^m\smallsetminus \{x\}$. Next assume that $x \notin \Lambda(\Gamma)$. Since $y \in Z_x$ if and only if $x \in Z_y$, the previous case implies that 
$$
\Lambda(\Gamma) \subset Z_x. 
$$
Then Theorem~\ref{thm:analytic envelopes} implies that $Z_x = \partial \Bb^m\smallsetminus \{x\}$. \hfill $\blacktriangleleft$

\medskip

Now for every complex affine line $L \subset \Cb^m$, the image 
$$
f\left( L \cap \overline{\Bb^m} \right)
$$
is contained in a complex affine line in $\Cb^M$. 

\medskip
\noindent \textbf{Claim:} $f$ is a rational function.

\medskip



\noindent \emph{Proof of Claim:} 
Fix a subset $J=\{ j_1 < \cdots < j_m\}$ of size $m$ in $\{1, \dots, M\}$. Then define $\pi : \mathbb{C}^M \to \mathbb{C}^m$ by $$ 
\pi(z_1, \dots, z_M) = (z_{j_1}, \dots, z_{j_m}). 
$$
Note that the map $\pi$ is linear so it maps complex affine lines to complex affine lines. Together with the previous claim, then for any complex affine line $L \subset \Cb^m$, the image
$\pi f(L \cap \Bb^m)$ is contained in a complex affine line.

Furthermore, since $f(\Bb^m)$ spans $\Cb^M$  the image $\pi f$ contains $(m + 1)$-points in general affine position. Otherwise, since $f(\vect{0})=\vect{0}$, the image of $\pi f$ would be contained in a proper linear subspace of $\Cb^m$, which would imply that $f(\Bb^m)$ is contained in a proper linear subspace of $\Cb^M$.   

Hence, by Theorem~\ref{thm:FTAG}, there exist a unital homomorphism $\tau : \Cb \rightarrow \Cb$ and 
$$g := \begin{pmatrix}
    A & b\\
    c^T & d
\end{pmatrix} \in \GL(m + 1, \Cb)$$
such that
$$\pi f(z) = \frac{A\tau_m(z) + b}{c^T\tau_m(z) + d}.
$$
Since $\tau(q) = q$ for all $q \in \Qb$, when $z \in \Qb^m \cap \Bb^m$ we have 
$$
\pi f(z) = \frac{Az + b}{c^Tz + d}.
$$
Hence by the identity theorem for holomorphic maps, 
$$
\pi f(z) = \frac{Az + b}{c^Tz + d}
$$
for all $z \in \Bb^m$.

Then the coordinate functions $f_{j_1},\dots, f_{j_m}$ are all rational. Since $J \subset \{1,\dots,M\}$ was an arbitrary subset of size $m$, this implies that every coordinate function of $f$ is rational, so $f$ is rational. \hfill $\blacktriangleleft$

\medskip

The Zariski density hypothesis implies that $\Gsf_f$ is non-compact (see Proposition~\ref{prop:properties Z dense} part (1)). Then, since $f$ is rational, \cite[Corollary 3.2]{DX2017} implies that there exist $\phi_1 \in \Aut(\Bb^M)$ and $\phi_2 \in \Aut(\Bb^m)$ such that  
\begin{align*}
\phi_1 \circ f \circ \phi_2(z) = (z,\vect{0})
\end{align*}
for all $z \in \Bb^m$. 
\end{proof}

\bibliographystyle{alpha}
\bibliography{SCV} 

\end{document}